\documentclass[12pt]{amsart}
\usepackage[left=2cm,top=2cm,right=3cm,nofoot]{geometry}
\usepackage{amsmath,mathtools}%
\usepackage{amsfonts}%
\usepackage{amssymb}%
\usepackage{graphicx}
\usepackage{esint}
\usepackage{hyperref}
\usepackage{enumitem}
\usepackage{accents}
\usepackage{etoolbox}
\patchcmd{\subsection}{-.5em}{.5em}{}{}
\newtheorem{theorem}{Theorem}[section]
\theoremstyle{plain}

\newtheorem{corollary}[theorem]{Corollary}

\newtheorem{lemma}[theorem]{Lemma}

\newtheorem{propo}[theorem]{Proposition}
\newtheorem{remark}[theorem]{Remark}

\numberwithin{equation}{section}
\theoremstyle{plain}

\usepackage{etoolbox}
\AtEndEnvironment{proof}{\setcounter{claim}{0}}
\newcommand{\re}{\mathbb{R}}

\begin{document}

\title[Nodal solutions to the Yamabe-type equations]{Nodal solutions of Yamabe- type equations on  positive Ricci curvature manifolds}

\author{Jurgen Julio-Batalla}
\address{Centro de Investigaci\'{o}n en Matem\'{a}ticas, CIMAT, Calle Jalisco s/n, 36023 Guanajuato, Guanajuato, M\'{e}xico}
\email{jurgen.julio@cimat.mx}

\author{Jimmy Petean}
\address{Centro de Investigaci\'{o}n en Matem\'{a}ticas, CIMAT, Calle Jalisco s/n, 36023 Guanajuato, Guanajuato, M\'{e}xico}
\email{jimmy@cimat.mx}

\begin{abstract} We consider a closed cohomogeneity one Riemannian manifold $(M^n,g) $ of dimension $n\geq 3$. If the Ricci curvature of $M$ is positive, we prove the existence of infinite nodal solutions for equations of the form   $-\Delta_g u + \lambda u = \lambda u^q$ with $\lambda >0$, $q>1$. In particular for a positive Einstein manifold which is of cohomogeneity one or 
fibers over a cohomogeniety one Einstein manifold we prove the existence of infinite nodal solutions for the Yamabe equation, with a prescribed number of connected components of its nodal domain.
\end{abstract}
\maketitle
\section{Introduction}

Let $(M^n ,g)$ be a closed Riemannian manifold of dimension $n\geq 3$. Denote by ${\bf s}_g$ the scalar curvature of
$g$. It is a classical problem, known as the {\it Yamabe problem}, to look for metrics of constant scalar curvature in
the conformal class of $g$. Given a metric $h$ conformal to $g$ we can write it as $h=u^{\frac{4}{n-2}} g$
(for a positive smooth function $u$ on $M$). Then the
scalar curvature of $h$ is given by ${\bf s}_h = u^{-\frac{n+2}{n-2}} (-a_n \Delta_g u + {\bf s}_g u )$, where $\Delta_g$ is
the Laplace operator and $a_n = \frac{4(n-1)}{n-2}$. Therefore ${\bf s}_h =\lambda \in \re$
if and only if $u$ is a (positive) solution of the {\it Yamabe equation}:
\begin{equation}\label{YE}
-a_n \Delta_g u + {\bf s}_g u = \lambda u^{\frac{n+2}{n-2}}.
\end{equation} 

\noindent
We let $p=p_n =\frac{n+2}{n-2}$. The exponent $p$ is the
critical Sobolev exponent. If replace $p$   by $q<p$ in (\ref{YE}) the equation is called {\it subcritical} and if we
replace it by $q>p$ it is called {\it supercritical}. It is a fundamental theorem that there is always a positive solution to the 
Yamabe equation (\ref{YE}).
This was  first stated by H. Yamabe in \cite{Yamabe}, but the proof contained a mistake that was fixed in several steps by N. Trudinger \cite{Trudinger}, T. Aubin \cite{Aubin},
and R. Schoen \cite{Schoen}. If the corresponding metric has non-positive constant scalar curvature (i.e. $\lambda \leq 0$)
it is easy to see that it is the unique solution of the problem. But if the metric has positive scalar curvature then there are in
general multiple solutions. It has been a fundamental problem to understand the space of positive solutions to the Yamabe 
equation in this case, and many multiplicity results have been obtained: see for instance \cite{Bettiol, Brendle, Piccione, Henry, Pollack, Schoen2}. Many of these multiplicity results 
are obtained on Riemannian products or more generally on certain Riemannian submersions by considering functions that are constant
along the fibers. In this case (\ref{YE}) reduces to a subcritical equation on the base. Supercritical equations are more difficult to
study and  have also been of great interest in analysis, see for instance \cite{CGM, CP, GMP}. 

\vspace{.3cm}

Many  results have been obtained recently about {\it nodal solutions}  of the equation (\ref{YE}). While it is known that
for any closed Riemannian manifold there is a positive solution of the Yamabe equation it is still not known if there is always also a nodal solution. But there are many results
in this direction. In particular, B. Ammann and E. Humbert  proved in \cite{Ammann} that under general situations (for instance
when the manifold is not locally conformally 
flat and its dimension is at least 11) there
is at least one nodal solution, which minimizes the second Yamabe invariant. 
In the particular case of the round sphere W. Ding constructed in \cite{Ding} infinite families of nodal solutions by using the symmetries of the sphere and variational methods. Similar results
on the round sphere were obtained in  \cite{JC} using ordinary differential equations  techniques similar to the ones we will use in this work. M. Clapp
and  J. C Fernandez in \cite{CF} also constructed infinite nodal solutions invariant under certain symmetries using variational techniques.  See also \cite{PMPP} and the references in these articles. 
An interesting family of examples are the positive Einstein metrics, since it is known by a theorem of M. Obata \cite{Obata} that a positive Einstein metric (different from the round metric on the sphere), is the only constant scalar curvature metric on its conformal class;  i. e. there is uniqueness of positive solutions to the Yamabe equation (and when the scalar curvature is positive this the only general situation where uniqueness 
is known to hold for positive solutions). 

\vspace{.3cm}

In this article we will  study nodal solutions of the  Yamabe-type equations 
\begin{equation}\label{YTE}
-\Delta_g u + \lambda u = \lambda u^q ,
\end{equation} 
where $\lambda \in \re_{>0}$ and $q>1$. If the scalar curvature of $g$, ${\bf s}_g$ is constant equal to $\frac{4(n-1)}{n-2} \lambda$ and
$q=p_n =  \frac{n+2}{n-2}$ is the critical Sobolev exponent, then (\ref{YTE}) is the Yamabe equation. And if $u$ is a 
positive solution then $u^{\frac{4}{n-2}} g$ also has constant scalar curvature ${\bf s}_g$. We will 
consider the case when
there is a subgroup $G$ of the isometry group of $g$ that acts with cohomogeneity one (i.e. the general orbits are hypersurfaces) and the Ricci curvature of $g$ is positive. In this situation we will prove the existence of nodal solutions of 
(\ref{YTE})  for any $\lambda >0$ and $q\in (1, p_G )$; here $p_G = \frac{n-m+2}{n-m-2}$, where $M$ is the dimension of $n$ and $m$ is the minimum of the
dimension of the orbits of $G$. Note that when $m>0$ we have that $p_G > p_n$ and (\ref{YTE}) is a supercritical equation. 

There are many examples in the  family of manifolds considered. If there is cohomogeneity one action by $G$ on 
the closed manifold $M$
the space of orbits $M/G$ could be $S^1$ or a closed interval. In the first case the fundamental group of $M$ is infinite
and $M$ cannot admit a metric of positive Ricci curvature,
but in the second case K. Grove and W. Ziller proved in \cite{Grove} the existence of an invariant metric of positive Ricci curvature. If $M/G = [0,1]$ and we call $\pi : M \rightarrow M/G$ the projection then there are two special orbits $M_1 = 
\pi^{-1} (0)$ and $M_2 = \pi^{-1} (1)$. We have that $m_i = dim (M_i ) < n-1$. If $m=\min \{ m_1 , m_2 \}$ then we
have called $p_G = \frac{n-m +2}{n-m -2}$; in particular if $m=n-2$ then $p_G = \infty$.  We will prove:

\begin{theorem} Let  $(M^n,g)$ be a closed Riemannian manifold of positive Ricci curvature. 
Assume that $G$ acts on $(M,g)$ with cohomogeneity one. If $q<p_G$ then for any $k\in\mathbb{N}$ there exists a solution $u_k$ of (\ref{YTE})  which is $G$-invariant and whose nodal set 
has exactly $k$ connected components. The nodal sets are $k$ copies of the regular fiber.
\end{theorem}

In particular this result applies to prove the existence of nodal solutions of the Yamabe equation on positive Einstein manifolds. The
case of the round metric on the sphere was considered in \cite{JC}. 
As a corollary of the theorem we have:

\begin{corollary}
Let $(M,g)$ be a closed positive Einstein manifold  with an isometric action of cohomogeneity one and without fixed points.
Then the Yamabe equation (\ref{YE}) on $(M,g)$ has infinite nodal solutions: for each positive integer $k$ there exists a solution 
of (\ref{YE}) whose nodal domain is given by exactly $k$ regular orbits of the action. 
\end{corollary}

There are many examples of cohomogeneity one positive Einstein manifolds without fixed points. For instance one has all the 
non-homogeneous examples constructed by C B\"{o}hm in \cite{Bohm}. Also if $(M,g)$ is homogeneous one can consider a
subgroup acting with cohomogeneity one. For instance for the particular case of the round sphere considered in \cite{JC} there are
many subgroups of $O(n+1)$ acting with cohomogeneity one. The simplest example is the action by $O(n)$ which fixes two
antipodal points. In this case the theorem cannot be applied and it is known that actually there are no nodal solutions which
are $O(n)$-invariant. But there are other subgroups of $O(n+1)$ which act with cohomogeneity one and with all orbits of 
positive dimension, as was studied in \cite{JC}.

The theorem can also be applied to Einstein manifolds of greater cohomogeneity. First one notes that if there is a Riemannian
submersion with totally geodesic fibers $\pi : (M,g) \rightarrow (N,h)$ then any function $u$ on $N$ gives a function $u \circ \pi$ on
$M$ which is constant along the fibers and the Laplacians commute: $\Delta_g (u\circ \pi )= (\Delta_h u) \circ \pi$. Then $u\circ \pi$ solves (\ref{YTE}) on $(M,g)$ if and only if $u$ solves (\ref{YTE}) on $(N,h)$. Then one can apply the theorem to $(N,h)$ in case it
is a cohomogeneity one Riemannian manifold of positive Ricci curvature. Note that if $m$ is a positive integer then $p_{m+n} <p_n$
and therefore the critical Yamabe equation on $(M,g)$ is a subcritical equation on $(N,h)$. Note also that the above applies in particular to the case of Riemannian products and one can apply the theorem if one of the factors has positive Ricci curvature and
a cohomogeneity one action. For instance M. Y. Wang and W. Ziller \cite[Corollary 2] {WZ} constructed positive Einstein metrics on total
spaces of Riemannian submersions over products of positive K\"{a}hler-Einstein manifolds. In this way they obtained positive Einstein manifolds of any given cohomogeneity to which one can apply the theorem, in the form of the following corollary: 

\begin{corollary} Let $(M,g)$ be a closed Riemannian manifold which fibers over a cohomogeneity one positive Einstein
manifold with totally geodesic fibers. Then the Yamabe equation (\ref{YE}) on $(M,g)$ has infinite nodal solutions: for any
positive integer $k$ there exists a nodal solution whose nodal domain has exactly $k$ connected components.
\end{corollary}

The proof of the theorem relies on the double shooting argument developed in \cite{JC} for the case of the round sphere.
By considering equation (\ref{YTE}) for $G$-invariant functions one reduces (\ref{YTE}) to an ordinary differential equation
defined on $M/G$ and with a singularity on each end-point. The initial (or {\it final}) value problem can be solved at each singularity. If two of these solutions match at a middle point we have a global solution for (\ref{YTE}). For the procedure to work it is
necessary that the mean curvature of the fibers (which are constant) give a monotone function on the interval.  We 
will point out that this is the case when the Ricci curvature is positive. In the argument in  \cite{JC} for the case 
of the round sphere it is used that all positive solutions of the Yamabe equation on the sphere are radial and
therefore there are no positive non-constant solutions invariant by other cohomogeneity one actions. But for a general
positive Ricci curvature metric with constant scalar curvature,  there might be non-constant positive solutions (for instance the are such solutions in Riemannian products of round spheres \cite{Petean}). We will overcome this problem by restricting the initial values
of the solutions of the ordinary differential equation. The proof of the theorem will be carried out in Section 3. In Section 2 we will see the
details of the reduction of (\ref{YTE}) to an ordinary differential equation.

\section{preliminaries: reduction to an ordinary differential equation}

We consider  a cohomogeneity one isometric action of a group $G$ on the Riemannian manifold $(M,g)$  so that
the orbit space $M/G$ is an interval. We let $M_1$ and $M_2$ be the singular orbits and we let $f:M \rightarrow 
[0,d]$ be $f(x) =d(x, M_1 )$, so $d=d(M_1 ,M_2 )$. Note that $f$ is a continuous function, and it is smooth on $M-(M_1 \cup M_2 )$. On a normal neighborhood $U_r$ around $M_1$, $U_r = \{ (x,y): x\in M_1 , y \in B(0,r) \subset  (T_x M_1 )^{\perp} \}$ the function
$f$ is given by $f(x,y) = \| y \| $. For $t \in (0,d)$ the mean curvature $h(t)$ of $f^{-1} (t)$ is given by $h(t) =  \Delta_g f \ 
(x)$, for any $x \in f^{-1} (t)$. We will need the information from the following elementary Proposition:

\begin{propo} Let $G$ act on $(M,g)$ so that $M/G$ is an interval. Let $M_1$ and $M_2$ be the singular orbits and 
$m_i = dim (M_i )$. Let $f: M \rightarrow [0,d]$ be the distance to $M_1$ and for $t\in (0,d)$ let $h(t)$ be the mean curvature
of $f^{-1} (t)$. Then 

\begin{equation}
\lim_{t\rightarrow 0}  t   \ h(t) = n-m_1 -1 \ \ \ ,  \ \ \ \ \lim_{t\rightarrow d} (t-d )   \ h(t) = n-m_2 -1.
\end{equation}

If the Ricci curvature of $g$ is positive then $h$ is a strictly decreasing function. So in particular there exists 
$t_0 \in (0,d)$ such that $h(t) > 0 $ if $t\in (0,t_0 )$ and $h(t)<0$ if $t\in (t_0 ,d )$.
\end{propo}

\begin{proof} The formula for the asymptotic behavior of $h$ is obtained by considering normal coordinates around
the singular orbits. In these coordinates one sees  $f^{-1}(t)$ as a $(n-m_i -1)$-sphere bundle over $M_i$.  Then one computes  the mean curvature of the hypersurfaces $f^{-1} (t)$ using vectors tangent to the spheres and vector tangent
to $M_i$  and the asymptotic behavior is easily checked. 

Now for $t\in (0,d)$ and any $x\in f^{-1}(t)$ we notice that $\| \nabla f (x) \| =1$ and $\nabla f (x)$ is normal to $f^{-1} (t)$. Let $L:=\nabla f$, so by the  radial curvature equation we have 
$$-\nabla_L S_t=S_t^2+R_L,$$
where $S_t$ is the shape operator of $f^{-1} (t)\;(S_t(\cdot)=\nabla_{(\cdot)}L)$ and $R_L(\cdot)=R(\cdot,L)L$ the curvature tensor of $(M,g)$.\\
The trace of this equation allows us to describe the monotonicity of the mean curvature function $(h(t)=\Delta f(x)=tr(S_t))$:
$$-(n-1)h^{\prime}(t)=-\nabla_Ltr(S_t)=-tr(\nabla_LS_t)=tr(S_t^2)+Ricc(L,L)\geq Ricc(L,L)>0.$$
Therefore $h$ is strictly decreasing.

\end{proof}

Now let $u:[0,d] \rightarrow \re$ and consider $U=u \circ f$. Note that $U \in C^2 (M)$ if and only if
$u\in C^2 [0,d]$ and $u'(0)=u'(d)=0$ (see for instance \cite[Lemma 2.3]{BJP}. We have that  
$$\Delta_g U = (u''  \circ f ) \  \| \nabla f \|^2 + (u' \circ f ) \  \Delta f= (u'' + u' h) \circ f .$$ 
Then $U$ solves (\ref{YTE}) if and only if $u$ solves

\begin{equation}\label{ODE}
u ''+ h u' + \lambda u  (|u|^{q-1}  - 1)=0.
\end{equation}

It follows that Theorem 1.1 is a consequence of the following:

\begin{theorem}\label{zeros} Let $h:(0,d) \rightarrow \re$ be a smooth function which verifies 

\begin{itemize}
\item $h'(t) < 0$ for all $t\in (0,d)$
\item $\lim_{t\rightarrow 0} th(t) =n-m_1 -1$, $\lim_{t\rightarrow d} (t-d)h(t)  =n-m_2 -1$, for some integers $0\leq m_1 ,
m_2 \leq n-2$, $n\geq 3$.
\end{itemize}

Let $m=\min \{ m_1 , m_2 \}$, and $p=\frac{n-m +2}{n-m-2}$. Then for any $q\in (1,p)$ and any positive 
integer $k$ equation
\ref{ODE} has a solution $u$ which has exactly $k$ zeroes in the interval $(0,d)$ and verifies $u'(0) =u'(d)=0$.

\end{theorem}

\section{Proof of Theorem 2.2}

For each $\alpha \in \re$ we consider the solution $u_{\alpha}$ of equation (\ref{ODE}) which satisfies 
$u_{\alpha}(0)=\alpha$, $u_{\alpha}'(0)=0$. It is well known that the initial value problem has a unique local solution, and
it depends continuously on $\alpha$ in the $C^1$-topology (it follows from a standard contraction argument,
see for instance \cite{k}).  

Note the there exists a unique $t_0 \in (0,d)$ such that $h(t) < 0 $ if $t\in (0,t_0 )$ and $h(t)>0$ if $t\in (t_0 ,d )$.

We consider the energy function 

$$E(u) =\frac{u'^2}{2} +\lambda \left( \frac{|u|^q}{q} -\frac{u^2}{2} \right) .$$

If $u_{\alpha}$ is a solution of (\ref{ODE}) then 

$$E(u_{\alpha} )' =  -h(t) ( u_{\alpha}' (t) )^2 .$$ 

It follows that $E(u_{\alpha})$ is decreasing in $[0,t_0 ]$ and increasing in $[t_0 , d]$. In particular this proves that $u_{\alpha}$ 
remains uniformly bounded in the $C^1$-topology on $[0,t_0 ]$ (while it is defined) and this implies that $u_{\alpha}$ is defined at least on
$[0,t_0 ]$.

We recall the following result from \cite[Theorem 3.1]{JC}

\begin{theorem}\label{Prop:prescribed zeroes} For constants  $A>0$, $p>1$ and $\lambda>0$ and a positive  $C^1$ function $H$ defined in the interval $[0,A]$, consider the initial condition problem

\begin{equation}\label{Eq:General singular}
\left\{\begin{tabular}{cc}
$w''(r) + \frac{H(r)}{r} w'(r) + \lambda( |w(r)|^{p-1} w -w) =0$ & in $[0,A]$\\
$w(0)=d, w'(0)=0,$ &
\end{tabular}\right.
\end{equation}

Suppose  the following inequality
\begin{equation}\label{Eq: Subcriticality ODE}
\frac{H(0) +1 }{2} < \frac{p+1}{p-1}
\end{equation}
holds true. Then, for any $\varepsilon >0$ and  any positive integer $k$ there exists $d_k > 0$ so that the solution $w_{d }$ of \eqref{Eq:General singular} has at least $k$ zeroes in $(0,\varepsilon )$ for any $d\geq d_k$
\end{theorem}

By calling $H(t)=th(t)$ equation (\ref{ODE}) falls into the previous result. Note that we have in this case $H(0) = n-m_1 -1$, by Proposition 2.1. So to apply the theorem we need

$$\frac{n-m_1 }{2} < \frac{q+1}{q-1} = 1 + \frac{2}{q-1},$$

or

$$\frac{q-1}{2} < \frac{2} {n-m_1 -2}$$

which happens if and only if 

$$q<  \frac{n-m_1 +2}{n-m_1 -2}.$$

\vspace{.3cm}

But since $m \leq m_1$ we have that $ \frac{n-m +2}{n-m -2} \leq  \frac{n-m_1 +2}{n-m_1 -2}  $, and we can apply the theorem to obtain:

\begin{lemma}\label{Z} For any $\varepsilon >0$ and  any positive integer $k$ there exists ${\alpha}_k > 0$ so that the solution $u_{\alpha }$ of (\ref{ODE}) has at least $k$ zeroes in $(0,\varepsilon )$ for any $\alpha \geq {\alpha}_k$.
\end{lemma}

We define the curve 
\begin{align*}
I:\mathbb{R}&\rightarrow\mathbb{R}^2\\
\alpha&\mapsto (u_{\alpha} (t_0),u'_{\alpha} (t_0)).
\end{align*}

Note that by the uniqueness of solutions of the ordinary differential equation with given initial conditions, $I$ is a simple curve. Also  $I(0)=(0,0)$, $I(-\alpha )=-I(\alpha )$ (if $u(t)$ is a solution of (\ref{ODE}) the $-u(t)$ is also a solution) and $I(1) = (1,0)$.
There is a well defined map ${\bf a} : (0,\infty ) \rightarrow \re^2 - \{(0,0) \}$ such that  ${\bf a} (1) =0$ and 
${\bf a} (\alpha ) $ gives an angle between $I(\alpha )$ and the positive real axis.

Note that $u_{\alpha} (t_0 ) =0$ if and only if $I(\alpha )$ lies in the imaginary axis if and only if ${\bf a} (\alpha ) = \frac{\pi}{2}
+ k \pi$ for some integer $k$. Similarly $u_{\alpha} '(t_0 ) =0$ if and only if ${\bf a} (\alpha ) =  k \pi$ for some integer $k$.

\vspace{.3cm}

For $\alpha \neq 0$ let $n(\alpha )$ be the number of zeroes of $u_{\alpha}$ in $(0, t_0 )$.

\begin{remark}\label{RemarkZeroes}
It is easy to see that as the curve $\alpha \mapsto I(\alpha )$ crosses the imaginary axis clockwise 
$n(\alpha )$  increases by one, while it decreases by one if $I(\alpha )$ crosses the imaginary axis counter-clockwise. It follows that ${\bf a}(\alpha )$ determines $n(\alpha )$. For instance if ${\bf a}(\alpha ) \in (-\pi /2 , \pi /2 )$ then $u_{\alpha}$ is
positive  in $(0, t_0 )$. If ${\bf a}(\alpha ) \in (-3\pi/2 ,-\pi /2 )$ then $u_{\alpha}$ has exactly one zero
in $(0, t_0 )$, and so on.
\end{remark}

The remark and Lemma \ref{Z} imply:

\begin{lemma}\label{alphainfinite} $\lim_{\alpha \rightarrow \infty} {\bf a} (\alpha ) = - \infty$.
\end{lemma}

Note that we can extend ${\bf a}$ to $(-\infty , 0)$ by ${\bf a} (-\alpha )={\bf a}( \alpha ) +\pi$ (${\bf a}$ still gives an angle between $I(\alpha )$ and the positive real axis), and we have:

\begin{lemma} $\lim_{\alpha \rightarrow -\infty} {\bf a}(\alpha ) = - \infty$.
\end{lemma}

The problem is similar if instead of the distance to $M_1$ one would have considered the distance to $M_2$. Namely, 
if $u$ is a solution to equation (\ref{ODE}) then $u^* (t)=u(d-t)$ would be the solution of the equation obtained by
interchanging  $M_1$ and $M_2$ (the function $h$ changes to $h^* (t) =-h(d-t)$).
Then the previous observations also apply if one instead considers the initial value problem at the endpoint $d$ of the interval of the definition of equation (\ref{ODE}).
Then one can define the curve 
\begin{align*}
F:\mathbb{R}&\rightarrow\mathbb{R}^2\\
\beta&\mapsto (u_{\beta}^* (t_0),(u^*_{\beta})' (t_0)),
\end{align*}

\noindent
where $u^*_{\beta}$ is the solution of (\ref{ODE}) such that
$u_{\beta}^* (d) =\beta$ and $(u^*_{\beta} )' (d) =0$. Then $F$ is a simple curve, $F(0)=(0,0)$, $F(-\beta )=-F(\beta )$, $F(1)=(1,0)$. We can define
an argument function ${\bf b} : (0,\infty ) \rightarrow  \re$ such that ${\bf b} (1) =0$ and ${\bf b}(\beta )$ gives an angle between $F(\beta )$ and the positive 
real axis. And define ${\bf b}$ on $(-\infty ,0) $ by ${\bf b}(-\beta ) = {\bf b}(\beta ) +\pi$. The previous results translate into

\begin{lemma} $\lim_{\beta \rightarrow \infty} {\bf b}(\beta ) = + \infty$.
\end{lemma}

\begin{lemma} $\lim_{\beta \rightarrow -\infty} {\bf b}(\beta ) = + \infty$.
\end{lemma}

To find solutions of (\ref{ODE})  we will look for values $\alpha, \beta$ such that $I(\alpha )=F(\beta )$. Then $u_{\alpha}$ and $u^*_{\beta}$ match at $t_0$ 
and give a solution defined on the whole interval $[0,d]$. Moreover we can see that by a careful analysis of the intersections one 
can know the number of zeroes of the corresponding solution. First we point out:

\begin{lemma}\label{alpha}
${\bf a}(\alpha ) <\pi /2$ for all $\alpha > 1$. 
\end{lemma}

\begin{proof} Assume that  the lemma is not true. Consider $\alpha_0 = \min \{ \alpha >1 : {\bf a} (\alpha ) = \pi /2 \}$.
Then we have $u_{\alpha_0 } (t_0 )=0$, $u_{\alpha_0 }' (t_0 ) > 0$. Then $u_{\alpha_0}$ must be
negative before $t_0$. But for $\alpha <\alpha_0$ $u_{\alpha}$ is positive in $[0,t_0 ]$ and this is a contradiction
since $\lim_{\alpha \rightarrow \alpha_0} u_{\alpha} = u_{\alpha_0}$.
\end{proof}

Similarly we have:

\begin{lemma}\label{beta}
${\bf b}(\beta ) > - \pi /2$ for all $\beta >1$. 
\end{lemma}

Then we consider  the simple curves
\begin{align*}
R:[1,\infty)&\rightarrow\mathbb{R}^2\\
\alpha &\mapsto ({\bf a}(\alpha ),|I(\alpha )|),\\
S:[1,\infty)&\rightarrow\mathbb{R}^2\\
\beta &\mapsto({\bf b}(\beta ),|F(\beta )|) :
\end{align*}

\noindent
they only can intersect at points $\alpha, \beta \geq1$ such that ${\bf a} (\alpha )={\bf b}(\beta )\in(-\pi/2,\pi/2).$ \medskip\\

Points of intersection of $I$ and $F$ give solutions to (\ref{ODE}). To understand the number of zeroes of such solutions
it will be useful to study the curves $R$ and $S$. We first note:

\begin{lemma}\label{positive}
If for some $\alpha , \beta >1$ we have that $R(\alpha )=S(\beta )$ then $I(\alpha ) =F(\beta )$ and the corresponding solution $u_{\alpha}=u^*_{\beta}$ is positive.
\end{lemma}

\begin{proof} The first statement is clear. Since ${\bf a}(\alpha ) > -\pi /2$ we have that $u_{\alpha}$ is positive
in $[0,t_0 ]$ (see Remark \ref{RemarkZeroes}). And since ${\bf b} (\beta ) < \pi /2$ we have that $u^*_{\beta}$ is positive in $[t_0 , d]$.

\end{proof}

With regards to nodal solutions we have:

\begin{lemma}\label{nodal} Let $k$ be a positive integer. 
If $R$ and $S-(2\pi k, 0)$ intersect, then the intersection gives a solution with exactly $2k$ zeroes.

If $R$ and $S-((2k-1)\pi , 0)$ intersect, then the intersection gives a solution with exactly $2k-1$ zeroes.
\end{lemma}

\begin{proof} Note that if for some  $\alpha , \beta >1$ we have that $R(\alpha )=S(\beta ) - (2k\pi , 0)$ then
$I(\alpha ) =F(\beta )$.  If $R(\alpha )=S(\beta ) - ((2k-1) \pi , 0)$ then
$I(\alpha ) =F(-\beta )$. In both cases we obtain a solution of equation (\ref{ODE}).

If $\alpha , \beta >1 $ are such that $R(\alpha ) =S(\beta )-(2k \pi ,0)$ (for some positive integer $k$) let
 $u=u_{\alpha} =u^*_{\beta} $  be the corresponding solution to (\ref{ODE}). The number of zeroes of $u$ in $[0,d]$
is given  by the number of zeroes of $u_{\alpha} $ restricted to $[0,t_0 ]$ plus the number of zeroes of $u_{\beta}^*$ restricted to $(t_0 , d]$.
The number of zeroes of $u_{\alpha} $ restricted to $[0,t_0 ]$ is determined by ${\bf a}(\alpha )$ while 
 the number of zeroes of $u_{\beta}^*$ restricted to $(t_0 , d]$ is determined by ${\bf b}(\beta )$. Since ${\bf a}(\alpha ) - 
{\bf b}(\beta ) =-2k\pi$ 
it follows as in Remark \ref{RemarkZeroes} (applied to ${\bf a}$ and ${\bf b}$)  that the total number of zeroes is exactly $2k$. 

If $R(\alpha )=S(\beta ) - ((2k-1) \pi , 0)$ let $u=u_{\alpha} =u^*_{-\beta}$  be the corresponding solution to (\ref{ODE}). The number of zeroes of $u$ in $[0,d]$
is given  by the number of zeroes of $u_{\alpha} $ restricted to $[0,t_0 ]$ plus the number of zeroes of $u_{-\beta}^*$ restricted to $(t_0 , d]$, which is of course equal to the number of zeroes of $u_{\beta}^* = -u_{-\beta}^*$ restricted to $(t_0 , d]$.
The number of zeroes of $u_{\alpha} $ restricted to $[0,t_0 ]$ is determined by ${\bf a}(\alpha )$ while 
 the number of zeroes of $u_{\beta}^*$ restricted to $(t_0 , d]$ is determined by ${\bf b}(\beta ) $. Since ${\bf a}(\alpha ) - 
{\bf b}(\beta ) =-(2k-1) \pi$ 
it follows as in Remark \ref{RemarkZeroes} that the sum is exactly $2k -1$. 
 
\end{proof}

Now for each $t\leq 0$ we consider the exit times 

$$e(t)=max\{ \alpha \in (0,\infty ) /\quad {\bf a}(\alpha )= t  \},$$

\noindent
which are well defined by Lemma  \ref{alphainfinite}. We will be particularly interested in $e_k =e(-k\pi )$. Let $C_k =\| I(e_k )\|$. 

If the constant solution $u=1$ is the only positive solution of  (\ref{ODE}), the curves $R$ and $S$ do not intersect 
by Lemma \ref{positive}, and together form a simple curve. 
If (\ref{ODE})  admits non-constant  positive solutions we can pick $\alpha^* ,\beta^*  > 1$ such that 
$R(\alpha^* )=S(\beta^* )$ and $R(\alpha )\neq S(\beta )$ for all $(\alpha ,\beta )\in( \alpha^* ,\infty)\times [1,\infty)$: this is well defined by Lemma \ref{alphainfinite} and Lemma \ref{beta} .\\
Henceforth the curve $$T:=R|_{[\alpha^* ,\infty)}\cup S|_{[ \beta^* ,\infty)} =R^* \cup S^*$$
is simple. Moreover, it follows that $T$ divides the half-plane $\mathbb{R}\times\mathbb{R}^+$ into two simply connected, disjoint regions $U$ and $D$ where the real axis is in the
closure of $D$. 

\begin{remark}\label{RemarkFinal} Consider the natural orientation of the curve $R^* = R|_{[ \alpha^* ,\infty)}$ and a tubular neighborhood $W$ around it. Note that $W - R^*$ has two connected components one to the right of $R^*$ (according to the natural orientation of the curve) and one to the left. It should be clear that the component to the right is contained in $U$ and the
component to the left is contained in $D$. In the next lines we will in particular prove this statement. Note that it cannot happen that
both components are contained in $U$ ( or $D$). For any $t\leq -\pi$ consider the vertical half-line $V_t =\{ t \} \times [0, \infty )$.
The intersection of $R^*$ with $V_t$ is a finite collection of points and closed intervals. We let $R(e(t))=(t,C(t) )$ be the
exit point of $R^*$ at $t$. Note that  $R^*$ can {\it touch} a  point in $V_t$ but remain in the same
side of $V_t$ before and after the intersection, or it can {\it cross} $V_t$. Suppose that $R^*$
crosses $V_t$ at $(t,a)$ in the direction of $-\infty$: i. e. for some $\alpha_1 >\alpha^*$ we have that $R(\alpha_1 )=(t, a)$ while
$R(\alpha_1  -\varepsilon) = (t',a')$ with $t' <t$ and $R(\alpha_1  +\varepsilon) = (t'',a'')$ with $t'' >t$. Assume $R^*$ crosses again  $V_t$ is at $(t,b)$, $R(\alpha_2 )= (t,b)$. Then  $R|_{[ \alpha_1,\alpha_2 ]}$ together with the segment $\{t \} \times [a,b]$ (or $\{t \} \times [a,b]$) is a closed curve $\gamma$. Note that $C(t) \notin [a,b]$ because
the curve $R|_{[e(t) ,\infty)}$ cannot intersect $\gamma$ and Lemma 3.4. It is then easy to see  that
$R^*$ crosses $V_t$ an even number of times under $C(t)$ and even number of times over $C(t)$. Since points
close to the real axis are in $D$ it is easy to see from this the statement mentioned at the beginning of the remark. It is
also easy to see that if $c\in (0, C(t) )$ and $(t,c) \in U$ then $(t,c)$ is in the vertical segment of a closed curve $\gamma$ as above.
And if $c>C(t)$ and $(t,c) \in D$ then again $(t,c)$ is in the vertical segment of a closed curve as above. This observation will
be used in the lemma below. 
\end{remark}

\begin{lemma}\label{monotone} The sequence $C_k = \| I(e_k ) \| $ is monotone.
\end{lemma}

\begin{proof}The curves $R$ and $R-(\pi ,0)$ are disjoint. Also $R-(\pi ,0)$ and $S$ are disjoint by Lemmas \ref{alpha} and \ref{beta}.
It follows that $R-(\pi ,0)$ and $T$ are disjoint. Therefore $R-(\pi ,0) \subset U $  or  $R-(\pi ,0) \subset D$. 
Note that the exit point of $R-(\pi ,0)$ at $-k\pi$ is $(-k\pi , C_{k-1} )$. 
If $R-(\pi ,0) \subset U $ we have that  $(-k\pi , C_{k-1} ) \in U$. Note also that after touching $(-k\pi , C_{k-1} )$ 
$R-(\pi ,0 )$ goes to $-\infty$ without intersecting $V_{-k\pi}$ again, so $(-k\pi , C_{k-1} )$ cannot be on 
a close curve $\gamma$ as in the previous remark.  But then the last part of the remark  implies that $C_k
\leq C_{k-1}$.  In the same way we prove that if  $R-(\pi ,0) \subset D$ then $C_k
\geq C_{k-1}$ and this proves the lemma.

\end{proof}

A similar result can be obtained in the same way for the sequences $C_k^*$ of exit levels of the curve $S$ at $k\pi$. 

The case when either $C_k$ or $C_k^*$  are decreasing can be solved in exactly the same way as was done for the particular
functions $h$ appearing in the 
case of the sphere in \cite[Lemma 4.7 and Appendix A]{JC}. We only give a brief description of the idea of the proof  for completeness. We  refer to
\cite{JC} for details.

\begin{lemma}
If the sequence $C_k$ is decreasing then for each $k$ there exists a solution of (\ref{ODE}) which has exactly $k$ zeroes,
all of them in $(0,t_0 )$. If $C_k^* $ is decreasing then for each $k$ there exists a solution of (\ref{ODE}) with exactly $k$ zeroes,
all of them in $(t_0 , d)$.

\end{lemma}

\begin{proof} One only needs to  prove the first statement, the second statement can be proved in the same way.  We first note
that $C_0 \leq 1$: if $C_0 >1$ then $u_{e_0} $ would have a local maximum at $t_0$ and therefore it must have a local minimum at some $t_1 <t_0$, and be increasing in $[t_1 , t_0 ]$. Let $e^+$ is the first time after 
$e_0$ when ${\bf a} (e^+ ) = -\pi /2$. Note that for $\alpha \in (e_0 , e^+ )$ ${\bf a} (\alpha ) \in (-\pi /2 ,0)$: this
means that  the solution $u_{\alpha}$ verifies that $u_{\alpha }(t_0 )>0$ and $u'_{\alpha }(t_0 )<0$. This (and the information
on $u_{e_0}$) implies that
$\alpha \in (e_0 , e^+ )$ the solution $u_{\alpha}$ has a local minimum and a local maximum in $(0,t_0 )$. But this implies
that $u_{e^+}$ would also have a local minimum in $(0,t_0 )$. But on the other hand, $u_{e^+} (t_0 )=0$  and therefore
its energy is positive at $t_0$. Since the energy is decreasing in $[0,t_0 ]$ this implies that it is positive in the interval.
But the energy is negative at a (positive) local minimum. This is a contradiction and therefore we have that $C_0 \leq 1$.
Then
we have that $C_i  <1 $ for all $i\geq 1$. Now, for $k\geq 1$ fixed, the solution $w_{e_k}$ must have exactly $k$ zeroes in
$(0,a_0 )$ and a local extrema at $a_0$ with value $C_k$ or $-C_k$. It follows that there must be other local extrema
between the last zero and $t_0$. For instance if $u_{e_k} (t_0 ) = C_k$ then $t_0$ is a positive local minimum of 
$u_{e_k}$ and it must have a local maximum before it. Note that the energy at $t_0$ is negative. 
Now we consider the values of $\alpha$ such that $u_{\alpha}$ has {\it the same configuration} (i. e. has $k$ zeroes and then
two consecutive local extrema).
As before consider 
$e^+$,  the first time after $e_k$ when the curve $I$ touches the $y$-axis.
Note that ${\bf a} (e^+ ) = {\bf a}  (e_k ) -\pi /2  =-k\pi - \pi /2$.
The solution $u_{e^+ }$ has exactly
$k$ zeroes in $(0,t_0 )$ and another one at $t_0$. Since the minimum of the energy $E(u_{e^+})$ is attained at $t_0$,
this means that the energy is positive in the whole interval $[0,d]$,
and therefore  $u_{e^+ }$ cannot have the previous configuration.  The values of $\alpha$ which have the configuration are open 
and by elementary argument one can see that at the first point when the configuration is lost one must have the last
extrema to move to $d$. The corresponding solution has exactly $k$ zeroes. And since the energy at $d$ is negative
and the energy is increasing in the interval $[t_0 , d]$ then it must be negative in the interval. Since the energy is positive at a
zero it follows then that all the zeroes must lie in
$(0,t_0 )$.

\end{proof}

With all the preliminary results obtained we are now in position to prove our  principal result: 

\begin{proof}(Theorem 2.2) It follows from the previous lemmas that if one of  the sequences $C_k$ or $C_k^*$ is not  
increasing then the theorem holds. So let us assume that both sequences are increasing. Write as before 
$T=R^* \cup S^*$. From Lemma \ref{monotone} (and its proof) we
have  that 
$R-(\pi ,0) \subset D$. And iterating the procedure we see that for any positive integer $k$, $R-k(\pi ,0) \subset D$.
A similar argument implies that  $S + k(\pi ,0) \subset D$. This implies that $S-k(\pi , 0)$ intersects $U$, the points in 
$S-k(\pi , 0)$ next to $+\infty$ are in $U$.
But this is saying that $T- k(\pi , 0 )$ has points in 
$U$ and points in $D$. It follows that  $T- k(\pi , 0 )$ must intersect $T$. Since $R^*  -k(\pi , 0)$ cannot intersect $T$ and
$S^* -k(\pi ,0 )$ cannot intersect $S$ it follows that $S^* -(k\pi , 0)$ must intersect $R^*$. Then Theorem 2.2 follows
from Lemma \ref{nodal}.
\end{proof}

\end{document}